\documentclass{article}
\usepackage{graphicx} 

\usepackage{amssymb}
\usepackage{booktabs}
\usepackage{longtable}
\usepackage{amsfonts}
\usepackage{pdflscape}
\usepackage{amscd}
\usepackage{caption}
\usepackage{array}
\usepackage{fancyhdr}
\usepackage{amsmath, amssymb}
\usepackage{stmaryrd} 
\usepackage{bbm}      
\usepackage{mathtools}
\usepackage{color} 
\usepackage{tikz-cd}
\usepackage[arrow, matrix, curve]{xy}
\usepackage{extarrows}
\usepackage{amsfonts}
\usepackage{amsthm}
\usepackage{hyperref}
\usepackage{cleveref}
\usepackage{graphicx}
\usepackage{comment}
\usepackage{comment}
\usepackage{MnSymbol}
\usepackage[numbers]{natbib}
\newcommand{\bs}{\mathbf{s}}
\newcommand{\bd}{\mathbf{d}}

\newcommand{\be}{\mathbf{e}}
\newcommand{\bn}{\mathbf{n}}

\newcommand{\bk}{\mathbf{k}}
\newcommand{\bbf}{\mathbf{f}}
\newcommand{\bg}{\mathbf{g}}

\DeclareMathOperator{\supp}{supp}
\DeclareMathOperator{\Hom}{Hom}

\DeclareMathOperator{\Exp}{Exp}
\DeclareMathOperator{\vol}{vol}

\DeclareMathOperator{\Log}{Log}
\DeclareMathOperator{\ad}{ad}

\newtheorem{theorem}{Theorem}[section]
\newtheorem{corollary}[theorem]{Corollary}
\newtheorem{lemma}[theorem]{Lemma}
\newtheorem{proposition}[theorem]{Proposition}
\newtheorem{conjecture}[theorem]{Conjecture}

\newtheorem{definition}[theorem]{Definition}
\newtheorem{remark}[theorem]{Remark}
\title{A nilpotent version of the first Kac conjecture}
\author{Lydia Gösmann}
\date{}

\begin{document}

\maketitle

\begin{abstract}
  We derive a formula for the nilpotent Kac polynomial and define an explicit generalized Kac-Moody Lie algebra (by generators and relations) and conjecture that its graded dimensions recover the constant terms of the nilpotent Kac polynomial. We outline empirical evidence supporting the conjecture and prove it for two special classes of quivers: quivers without oriented cycles except loops and quivers with "negative enough" Euler form.
\end{abstract}

\section{Introduction}
For a dimension vector $\bd$ of a quiver $Q$ let $A_{\bd}$ be the Kac polynomial attached to Q and $\bd$. For a finite field $\mathbb{F}_q$ the integer $A_{\bd}(q)$ is equal to the number $A_{\bd}(\mathbb{F}_q)$ of isomorphism classes of absolutely indecomposable representations of $Q$ of dimension vector $\bd$ over $\mathbb{F}_q$. The polynomial $A_{\bd}(t) \in \mathbb{Z}[t]$ is monic and of degree $1-\langle \bd, \bd \rangle$ and has nonnegative coefficients. \\
Furthermore, one can also ask for the number of absolutely indecomposable nilpotent representations. This question is addressed in work of Bozec-Schiffmann-Vasserot \cite{pointsnilpotentvarities}. There they developed a formula for the nilpotent and 1-nilpotent Kac polynomial counting nilpotent or 1-nilpotent absolutely indecomposable representations, respectively. In particular, they prove that the constant term of the 1-nilpotent Kac polynomial is the dimension of a generalized Kac-Moody algebra in the corresponding degree. The question of a Lie algebra corresponding to the nilpotent Kac polynomial remained open. In this paper we derive an explicit formula for the nilpotent Kac polynomial and propose an explicit presentation (by generators and relations) of a Lie algebra $\mathfrak{g}_Q^{nil}$ and conjecture that its graded dimensions recover the constant terms of the nilpotent Kac polynomial. We give computational evidence, prove consistency with the generalized Kac–Moody case as well as the special case of quivers with negative enough Euler form. Moreover, we observe that the generators are governed by the Le Bruyn–Procesi classification of dimension vectors of simple-representation and we interpret the conjecture in the context of BPS theory. A proof of the conjecture will appear in a forthcoming joint paper with L. Hennecart.

\section*{Acknowledgments}
I would like to thank Ben Davison and Lucien Hennecart for explaining their paper \cite{BPS} and pointing out the connection to BPS theory. Moreover, I thank my supervisor Markus Reineke for introducing me to this topic and for constant support and encouragement. 
The work was supported by the DFG GRK 2240 “Algebro-Geometric Methods in Algebra, Arithmetic and Topology”.

\section{Nilpotent Kac polynomial}

$Q$ is a finite quiver with a set of vertices $Q_0$, a set of arrows $Q_1$ and arrows written $\alpha:i \rightarrow j$. We will consider representations of $Q$ of some dimension vector $\bd \in \mathbb{N}Q$. We denote by $I$ the set of dimension vectors $\bs_i$ having dimension $1$ only at vertex $i$ and $0$ everywhere else.
We recall that the (in general non-symmetric) Euler form is given by $$\langle \be, \bd\rangle =\sum_{i \in Q_0} e_id_i- \sum_{\alpha: i \rightarrow j} e_id_j~.$$ 
We denote the symmetrized Euler form by $$(\be,\bd)= \langle \be,\bd \rangle +\langle \bd,\be \rangle~.$$
For a fixed dimension vector $\bd$ we define the space of representations as 
$$R_{\bd}=R_{\bd}(Q)=\bigoplus_{\alpha: i \rightarrow j} \Hom(\mathbb{C}^{d_i}, \mathbb{C}^{d_j})$$
acted on by the base change group
$$G_{\bd} = \prod_{i \in Q_0} GL(\mathbb{C}^{d_i}).$$
Let $N_{\bd}$ be the nullcone of this action, that are all representations $V$ such that $0$ is contained in the closure of the orbit $G_{\bd} \cdot V$ or equivalently all operators $V_{\alpha_1} \circ \dots \circ V_{\alpha_s}$ for oriented cycles $\alpha_1 \cdots \alpha_s$ in Q are nilpotent \cite[section 1.4]{optimalfiltrations}.\\
We denote by $K_0(Var_{\mathbb{C}})$ the Grothendieck ring of complex varieties; the class of a variety $X$
therein is denoted by $[X]$.  We define the Pochhammer symbol $(q)_n = (1 - q) \cdots (1 - q^n)$ and a shorthand notation concerning dimension vectors 
\begin{align*}
 \left[{{\bf d}\atop{{\bf d}^1,\ldots,{\bf d}^s}}\right]_q = \prod_i \frac{(q)_{d_i}}{\prod^s_{k=1} (q)_{d_i^k}}   
\end{align*}
for $\bd = \bd^1 + \dots + \bd^s$.\\
We define $n_{\bd}(q) = |N_{\bd}(\mathbb{F}_q)|$ as the number of nilpotent representations over $\mathbb{F}_q$.

In our recent paper \cite{motivenullcone} we have developed several formulas for the motive of the nullcone of a quiver, in particular in Corollary 3.3. the formula
\begin{align} \label{motive nullcone}
[N_{\bf d}]=(-1)^{\dim {\bf d}}\cdot\mathbb{L}^{\sum_i{d_i\choose 2}}\cdot\sum_{{\bf d}^*}(-1)^s\cdot\mathbb{L}^{-\sum_{k<l}\langle{\bf d}^k,{\bf d}^l\rangle}\cdot\left[{{\bf d}\atop{{\bf d}^1,\ldots,{\bf d}^s}}\right]_\mathbb{L},
\end{align}
where the sum ranges over all ordered decompositions ${\bf d}={\bf d}^1+\ldots+{\bf d}^s$ of ${\bf d}$ into non-zero parts ${\bf d}^k$. 
\begin{corollary}
    We have
    \begin{align*}
     n_{\bd}(q) =(-1)^{\dim {\bf d}}\cdot q^{\sum_i{d_i\choose 2}}\cdot\sum_{{\bf d}^*}(-1)^s\cdot q^{-\sum_{k<l}\langle{\bf d}^k,{\bf d}^l\rangle}\cdot\left[{{\bf d}\atop{{\bf d}^1,\ldots,{\bf d}^s}}\right]_q 
    \end{align*}
    is a counting polynomial in $q$.
\end{corollary}
\begin{proof}
In the paper the formula is stated for characteristic $0$, but it is still valid as a point counting formula in $\mathbb{F}_q$ for the following reason: The assumption of characteristic $0$ is necessary in the proof of theorem 2.2 to conclude that the bijection of motives is indeed an isomorphism. Nevertheless, for counting points over a finite field, a bijection between the motives is sufficient, thus the formula above also provides an explicit expression for the counting polynomial.
\end{proof}
In \cite{pointsnilpotentvarities} the nilpotent Kac polynomial is expressed as
\begin{align*}
    \sum_{\bd} \vol(\underline{C}^{nil}_{\bd, \mathbb{F}_q}) z^{\bd} = \exp (\sum_{\bd} \sum_{l \geq 1} \frac{A^{nil}_{\bd}(q^l)}{q^l-1}z^{l\bd}/l) = \Exp (\sum_{\bd} \frac{A^{nil}_{\bd}(q)}{q-1}z^{\bd})
\end{align*}
and moreover 
\begin{align*}
    \vol(\underline{C}^{nil}_{\bd, \mathbb{F}_q}) = \sum_{\underline{\bd}} q^{o(\underline{\bd})} \prod_i \vol (\underline{C}_{\bd_i, \mathbb{F}_q}), 
\end{align*}
where the sum ranges over all tuples $\underline{\bd}= (\bd_1, \dots, \bd_s)$ with $\sum_i i\bd_i =\bd$, meaning over all partitions of $\bd$ and 
\begin{align*}
    o(\underline{\bd}) = - \sum_i (i-1) \langle\bd_i, \bd_i \rangle - \sum_{i < j }i (\bd_i, \bd_j) ~.
\end{align*}
Here, $\vol (\underline{C}_{\bd_i, \mathbb{F}_q})$ is the volume of the nilpotent representation stack with dimension vector $\bd_i$, which is the same as the number of nilpotent representaations over $\mathbb{F}_q$ divided by the number of elements in the group $GL_{\bd}$ over $\mathbb{F}_q$:
\begin{align}\label{volume nilp repres stack}
    \vol(\underline{C}^{nil}_{\bd, \mathbb{F}_q}) &= \sum_{\underline{\bd}} q^{o(\underline{\bd})} \prod_i \frac{n_{\bd_i}}{(-1)^n q^{\bd_i(\bd_i-1)/2}(q)_{\bd_i}}
\end{align}
Combining all these formulas and using the plethystic logarithm we get an explicit formula for the nilpotent Kac polynomial
\begin{align*}
  \sum_{\bd}  \frac{A_{\bd}^{nil}(q)}{q-1}z^{\bd} &= \Log (\sum_{\bd} \vol(\underline{C}^{nil}_{\bd, \mathbb{F}_q}) z^{\bd}) = \Log (\sum_{\bd} \sum_{\underline{\bd}} q^{o(\underline{\bd})} \prod_k \frac{n_{\bd_k}}{(-1)^n q^{\bd_k(\bd_k-1)/2}(q)_{\bd_k}} z^{\bd}) \\
  &= \psi^{-1} \circ \log (1+\sum_{\bd >0} \sum_{\underline{\bd}} q^{o(\underline{\bd})} \prod_k \frac{n_{\bd_k}}{(-1)^n q^{\bd_k(\bd_k-1)/2}(q)_{\bd_k}} z^{\bd}) \\
  &= \psi^{-1} (\sum_{i \geq 1} (-1)^{i-1} \frac{1}{i} (\sum_{\bd >0} \sum_{\underline{\bd}} q^{o(\underline{\bd})} \prod_k \frac{n_{\bd_k}}{(-1)^n q^{\bd_k(\bd_k-1)/2}(q)_{\bd_k}} z^{\bd})^i) \\
  &= \sum_{j\geq 1} \frac{\mu(j)}{j} \psi_j(\sum_{i \geq 1} (-1)^{i-1} \frac{1}{i} (\sum_{\bd >0} \sum_{\underline{\bd}} q^{o(\underline{\bd})} \prod_k \frac{n_{\bd_k}}{(-1)^n q^{\bd_k(\bd_k-1)/2}(q)_{\bd_k}} z^{\bd})^i)
  \end{align*}
  Replacing the motive of $N_{\bd}$ with formula \ref{motive nullcone} yields a formula
  \begin{align*}
  &\sum_{j\geq 1} \frac{\mu(j)}{j} \psi_j(\sum_{i \geq 1} (-1)^{i-1} \frac{1}{i} (\sum_{\bd >0} \sum_{\underline{\bd}} q^{o(\underline{\bd})} \cdot \\
  &\prod_k \frac{(-1)^{\dim {\bf d_k}}\cdot  q^{\sum_i{d_{k_i}\choose 2}}\cdot\sum_{\bf d_k^*}(-1)^s\cdot q^{-\sum_{m<l}\langle{\bf d_k}^m,{\bf d_k}^l\rangle}\cdot\left[{{\bf d_k}\atop{{\bf d_k}^1,\ldots,{\bf d_k}^s}}\right]_q}{(-1)^n q^{\bd_k(\bd_k-1)/2}(q)_{\bd_k}} z^{\bd})^i)~,
\end{align*}    
which implemented in Julia calculates the nilpotent Kac polynomial for a given quiver $Q$ and a given dimension vector $\bd$. The formula above is what the Julia implementation evaluates to produce the values against which conjecture 3.4 is tested.

\section{Conjecture: Nilpotent Kac-Moody algebra}

 \begin{definition} \label{def nilp Kac-moody}
For a quiver $Q$ the (positive part) of the nilpotent Kac-Moody algebra $\mathfrak{g}^{nil}_Q$ is given by the following generators:
\begin{itemize}
    \item $e_{\bs_i}$ for $\bs_i \in I$, if the vertex $i$ has no loops, and 
     \item $e_\bd$ for any dimension vector $\bd$ such that $\supp(\bd)$ is a strongly connected subquiver and $\langle \bd, \bs_i \rangle \leq 0$ and $\langle \bs_i, \bd \rangle \leq 0$ for all $1 \leq i \leq n$.
\end{itemize}
The generators are subject to the following relation:
\begin{itemize}
    \item $[e_{\bbf}, e_{\bg}] =0$ for every pair of generators with $(\bbf, \bg)=0$ and
   \item $(\ad e_{\bs_i})^{1-(\bs_i, \bbf)} e_{\bbf} = 0$, for $\bbf \neq \bs_i$, if $\bs_i \in I$ and the vertex $i$ has no loops.
\end{itemize}
 \end{definition}

\begin{remark}
The construction above defines (the positive part of) a Borcherds algebra \cite{genkacmoody}. The size of the corresponding Cartan square matrix is the number of generators of the above construction, that means all dimension vectors satisfying the conditions above, so it is typically infinite. Moreover, the entry $c_{\bd, \be}$ is the symmetrized Euler form $(\bd,\be)$. If $\bd \neq \be$ then $c_{\bd, \be} \leq 0$ and $c_{\bs_i, \bs_i}=2$, if $\bs_i \in I$ and the vertex $i$ has no loops. All other $c_{\bd, \bd}$ have an value $\leq 0$. A direct calculation shows, that the construction of the Borcherds algebra with this Cartan matrix, reconstruct exactly the above relations.
\end{remark}

A program in Julia has been developed counting the generators of the nilpotent Kac-Moody algebra in a given degree. Various verifications form the empirical basis of the following conjecture.  

\begin{conjecture}\label{conjecture}
   For any quiver $Q$ and any dimension vector $\bd$ we have
\begin{align*}
    \dim \mathfrak{g}_Q^{nil} [\bd] = A^{nil}_{\bd}(0) ~.
\end{align*} 
\end{conjecture}
To be more precise, here is a list of all quivers and dimension vectors, denoted by adjacency matrices, that have been tested: \\
\\
The following quivers with $1$ vertex were tested up to dimension vector/degree $(10)$:\\
$[0]$, $[1]$, $[2]$, $[3]$, $[4]$, $[5]$.\\
 
The following quivers with $2$ vertices were tested up to dimension vector/degree $(5,5)$:\\
$[1\,1;\,1\,0]$, $[0\,2;\,2\,0]$, $[1\,2;\,1\,0]$, $[3\,2;\,1\,0]$, $[1\,2;\,2\,0]$, $[1\,2;\,3\,0]$, $[0\,2;\,1\,0]$, $[0\,1;\,1\,0]$, $[0\,1;\,0\,0]$, $[0\,3;\,1\,0]$, $[1\,2;\,2\,1]$, $[0\,3;\,2\,0]$, $[0\,3;\,3\,0]$, $[1\,1;\,0\,1]$, $[1\,1;\,1\,1]$, $[2\,1;\,1\,2]$, $[1\,1;\,0\,0]$, $[1\,0;\,0\,1]$, $[2\,1;\,1\,1]$.\\
 \\
The following quivers with $3$ vertices have been tested up to dimension vector/degree $(3,3,3)$:\\
$[1\,1\,0;\,0\,1\,1;\,0\,0\,1]$, $[1\,1\,0;\,0\,1\,1;\,0\,1\,1]$, $[1\,1\,0;\,1\,1\,1;\,0\,1\,1]$, $[1\,1\,1;\,1\,1\,1;\,0\,1\,1]$, \\
$[1\,1\,1;\,1\,1\,1;\,1\,1\,1]$, $[2\,1\,0;\,0\,2\,1;\,0\,0\,2]$, $[2\,1\,0;\,0\,2\,1;\,0\,1\,2]$, $[2\,1\,0;\,1\,2\,1;\,0\,1\,2]$,\\
$[2\,1\,1;\,1\,2\,1;\,1\,1\,2]$, $[2\,1\,0;\,1\,1\,1;\,0\,1\,2]$, $[2\,1\,0;\,1\,2\,1;\,0\,1\,1]$, $[1\,1\,0;\,0\,2\,1;\,0\,1\,1]$,\\
$[1\,1\,0;\,0\,1\,1;\,0\,1\,2]$, $[1\,1\,0;\,1\,2\,1;\,0\,1\,1]$.\\
 \\
The following quivers with $4$ vertices have been tested up to dimension vector/degree $(2,2,2,2)$:\\
$[1\,1\,0\,0;\,0\,1\,1\,0;\,0\,0\,1\,1;\,0\,0\,0\,1]$,
$[1\,1\,0\,0;\,1\,1\,1\,0;\,0\,0\,1\,1;\,0\,0\,0\,1]$,\\
$[1\,1\,0\,0;\,1\,1\,1\,0;\,0\,1\,1\,1;\,0\,0\,0\,1]$, 
$[1\,1\,0\,0;\,1\,1\,1\,0;\,0\,1\,1\,1;\,0\,0\,1\,1]$, \\
$[1\,1\,0\,0;\,1\,1\,0\,0;\,0\,0\,1\,1;\,0\,0\,1\,1]$, 
$[1\,1\,0\,0;\,1\,1\,1\,0;\,0\,0\,1\,1;\,0\,0\,1\,1]$, \\
$[0\,1\,0\,0;\,1\,0\,0\,0;\,0\,0\,0\,1;\,0\,0\,1\,0]$, 
$[0\,1\,0\,0;\,1\,0\,1\,0;\,0\,0\,0\,1;\,0\,0\,1\,0]$, \\
$[0\,1\,0\,0;\,1\,0\,1\,0;\,0\,1\,0\,1;\,0\,0\,1\,0]$.\\
 \\
The following quiver with $5$ vertices has been tested up to dimension vector/degree $(1,1,1,1,1)$:\\
$[0\,1\,0\,0\,0;\,0\,0\,1\,0\,0;\,1\,0\,0\,0\,0;\,0\,0\,0\,0\,1;\,0\,0\,0\,1\,0]$.\\
\\
The corresponding values are attached in the appendix.
 
\section{Special cases}
\subsection{Comparison to generalized Kac-Moody algebra}

We recall the definition of the generalized Kac-Moody algebra from \cite{pointsnilpotentvarities}:
 \begin{definition} \label{def gen KM from BSV}
     The (positive part) of the generalized Kac-Moody algebra is given by the following generators 
\begin{itemize}
      \item $e_{\bs_i}$ for $\bs_i \in I$, if the vertex $i$ has no loops and
     \item $e_{\bk_i}$ for $\bk_i \in K_i$,if the vertex $i$ has at least one loop.
     \end{itemize} 
     The generators are subject to the following relation:
\begin{itemize}
    \item $[e_{\bbf}, e_{\bg}] =0$ for every pair of generators with $(\bbf, \bg)=0$ and
    \item $(\ad e_{\bs_i})^{1-(\bs_i, \bbf)} e_{\bbf} = 0$ for $\bbf \neq \bs_i$, if $\bs_i \in I$ and the vertex $i$ has no loops.
\end{itemize}
 \end{definition}

Since the notion of 1-nilpotency and nilpotency coincide, if the quiver has no oriented cycles except loops, we have the following propostion as a consistency check.

\begin{proposition}
    The construction from Definition \ref{def gen KM from BSV} and Definition \ref{def nilp Kac-moody} coincide if the quiver $Q$ has no oriented cycles except loops.
\end{proposition}
\begin{proof}
    Both constructions have the generators $e_{\bs_i}$ for all $\bs_i \in I$. Moreover, since $Q$ has no oriented cycles except loops, $\supp(\bd)$ is only a strongly connected subquiver, if $\bd$ has only one nonzero entry. If this entry is at vertex $i$ and the vertex $i$ has at least one loop, we have $\langle \bd, \bs_i \rangle \leq 0$ and $\langle \bs_i, \bd \rangle \leq 0$ for all $\bs_i \in I$. If this entry is at vertex $j$ and vertex $j$ has no loops, then $\langle \bd, \bs_j \rangle > 0$ and $\langle \bs_j, \bd \rangle > 0$. This shows that $\mathfrak{g}_Q^{nil}$ has the same generators as the generalized Kac-Moody algebra $\mathfrak{g}_Q$, if $Q$ has no oriented cycles except loops. Since both algebras have also the same relations, they coincide.
\end{proof}

\subsection{Negative Enough Euler form}
 There is another special case, where we can prove the conjecture directly. We say a quiver has a negative enough Eulfer form, if $\langle \bd, \be \rangle \leq 0$, with equality only if $\bd=0$ or $\be=0.$ This means that each vertex has at least two loops, and that there is an arrow from $i$ to $j$ for every vertices $i$ and $j$. In this situation the construction of the nilpotent Kac-Moody algebra consists of generators for every dimension vector, since every subquiver is strongly connected and the Euler form is always negative for non-zero dimension vectors. Moreover, there are also no relations between the generators. So the dimension of the nilpotent Kac-algebra is given by multidimensional Lyndon words counts as well as the constant term of the nilpotent Kac-polynmial as the following lemma proves.

 \begin{lemma}
     If the Euler form of the quiver $Q$ is negative enough, we have an identity
     \begin{align*}
         \sum_{d \neq 0} A^{nil}_d(0) z^{\bd} =  \sum_{\bd \neq 0} \sum_{e | \gcd (d_1, \dots, d_r} \frac{\mu(e)}{e} \sum_{k \geq 1} \frac{1}{k} \sum^k_{j=1} (-1)^{k-j} \binom{k}{j} \prod_i \binom{\frac{d_i}{e}+j-1}{j-1} z^{\bd},
     \end{align*}
     where $z^{\bd} = z_1^{d_1} \cdots z_r^{d_r}$.
 \end{lemma}

 \begin{proof}
   We start with the identity 
   \begin{align*}
       \sum_{\bd} \frac{[N_{\bd}]}{[G_{\bd}]} t^{\bd} \cdot \sum_{\be} \frac{t^{\be}}{(\mathbb{L)_{\be}}}=1
   \end{align*}
   in $\mathbb{M}_{\mathbb{L}}[[Q_0]]$ from \cite{motivenullcone}, where $\mathbb{M}_{\mathbb{L}} [[Q_0]]$ is the twisted formal power series ring with topological $\mathbb{M}$-baasis elements $t^{\bd}$ for $\bd \in \mathbb{N}Q_0$ and multiplication $t^{\bd} \cdot t^{\be} = \mathbb{L}^{-\langle \bd, \be \rangle} \cdot t^{d+e}$. For $\mathbb{L} =0$ the only terms surviving correspond to $\bd=0$ or $\be =0$, since the Euler form is negative enough. After comparing the coefficients we get:
   \begin{align*}
       \frac{[N_{\bd}]}{[G_{\bd}]} = - \frac{1}{(\mathbb{L})_{\bd}} = -1
   \end{align*}
   for $\bd \neq 0$ and $ \frac{[N_{\bd}]}{[G_{\bd}]}=0$ for $\bd =0$.
   Now we consider the formula for the volume of the nilpotent representation stack \ref{volume nilp repres stack} and determine when the exponent of $q$ is zero. This is exactly the case if the partition $\underline{\bd}$ of $\bd$ contains only one entry: $\bd$ itself. So altogether
   \begin{align*}
       \sum_{\bd \geq 0} \vol(\underline{C}^{nil}_{d,  \mathbb{F}_q})_{|0} z^{\bd} &= 1 + \sum_{d > 0} \frac{[N_{\bd}]}{[G_{\bd}]} z^{\bd} = 1 - \sum_{d >0} z^{\bd} = 1 - (\prod_i \frac{1}{1-z_i}-1) \\
       &= 2 - \prod_i \frac{1}{1-z_i} = \frac{2 \prod_i (1-z_i)-1}{\prod_i (1-z_i)}~.
   \end{align*}
   Since 
   \begin{align*}
       \sum_{\bd \neq 0}  \vol(\underline{C}^{nil}_{d,  \mathbb{F}_q})_{|0} z^{\bd} = \Exp (\sum_{\bd} \frac{A^{nil}_d(0)}{q-1}z^{\bd}) = \Exp  (-\sum_{\bd \neq 0} A^{nil}_d(0) \cdot z^{\bd})
   \end{align*}
   we get 
   \begin{align*}
     \sum_{\bd \neq 0} A_{\bd}^{nil}(0) z^{\bd} = \Log (\frac{\prod_i (1-z_i)}{2 \prod_i (1-z_i)-1})~.
   \end{align*}
The right hand side coincides with the formula for the multidimensional Lyndon words counts, since
\begin{align*}
    \Log (\frac{\prod_i (1-z_i)}{2 \prod_i (1-z_i)-1}) = \psi^{-1}  \log (\frac{\prod_i (1-z_i)}{2 \prod_i (1-z_i)-1})) = \psi^{-1} (- \log (2- \prod_i \frac{1}{1-z_i}))~,
\end{align*}
where we divide numerator and denominator by $\prod_i (1-z_i)$ in the last step.
Now we use the identity $\sum_{k \geq 1} \frac{u^k}{k} = -\log (1-u)$ for $u= \prod_i \frac{1}{1 - z_i} -1$:
\begin{align*}
    \psi^{-1} (- \log (2- \prod_i \frac{1}{1-z_i})) = \psi^{-1} (\sum_{k \geq 1} \frac{1}{k} (\prod_i \frac{1}{1 - z_i} -1)^k) ~.
\end{align*}
Next, we use the binomial theorem $\sum^k_{j=0} (-1)^{k-j} \binom{k}{j} y^j = (y-1)^k$ with $y=\prod_i \frac{1}{1 - z_i}$:
\begin{align*}
     \psi^{-1} (\sum_{k \geq 1} \frac{1}{k} (\prod_i \frac{1}{1 - z_i} -1)^k)  = \psi^{-1}(\sum_{k \geq 1} \frac{1}{k} \sum^k_{j=0} (-1)^{k-j} \binom{k}{j} \prod_i \frac{1}{(1 - z_i)^j}) ~.
\end{align*}
Furthermore,
\begin{align*}
    \sum^k_{j=0} (-1)^{k-j} \binom{k}{j} \prod_i \frac{1}{(1 - z_i)^j}) &=  \sum^k_{j=1} (-1)^{k-j} \binom{k}{j} \prod_i \frac{1}{(1 - z_i)^j}) + (-1)^k \\
    &= \sum^k_{j=0} (-1)^{k-j} \binom{k}{j} \prod_i \frac{1}{(1 - z_i)^j}) - \sum_{j=0}^k (-1)^{k-j} \binom{k}{j} \\
    &= \sum^k_{j=0} (-1)^{k-j} \binom{k}{j} (\prod_i \frac{1}{(1 - z_i)^j})-1) \\
    &= \sum^k_{j=0} (-1)^{k-j} \binom{k}{j} \sum_{ \bd \neq 0} \prod_i \binom{d_i+j-1}{j-1} z^{\bd}\\
    &= \sum_{\bd \neq 0} \sum^k_{j=1} (-1)^{k-j} \binom{k}{j} \prod_i \binom{d_i +j -1}{j-1} z^{\bd}  ~.
\end{align*}
This yields
\begin{align*}
   & \psi^{-1}(\sum_{k \geq 1} \frac{1}{k} \sum^k_{j=0} (-1)^{k-j} \binom{k}{j} \prod_i \frac{1}{(1 - z_i)^j}) \\
    &= \psi^{-1}(\sum_{\bd\neq 0}\sum_{k \geq 1} \frac{1}{k} \sum^k_{j=1} (-1)^{k-j} \binom{k}{j} \prod_i \binom{d_i+j-1}{j-1} z^{\bd})\\
    &= \sum_{\bd \neq 0} \psi^{-1}(\sum_{k \geq 1} \frac{1}{k} \sum^k_{j=1} (-1)^{k-j} \binom{k}{j} \prod_i \binom{d_i+j-1}{j-1} z^{\bd}) \\
    &= \sum_{\bn \neq 0} \sum_{e | \gcd (n_1, \dots, n_r} \frac{\mu(e)}{e} \sum_{k \geq 1} \frac{1}{k} \sum^k_{j=1} (-1)^{k-j} \binom{k}{j} \prod_i \binom{\frac{n_i}{e}+j-1}{j-1} z^{\bn} ~.
\end{align*}

 \end{proof}

\begin{remark}
    For Corollary 3.5. in \cite{motivenullcone} the condition that the Euler form is negative enough is also necessary.  
\end{remark}

\section{Dimension vectors for simple representations}
LeBruyn and Procesi \cite{simplerepresentation} have characterized dimension vectors of simple representations with the following theorem and since the above construction bears a close resemblance to their description, it is worth comparing the two in detail.

\begin{theorem}
$\bd \in \mathbb{N}^n$ is the dimension vector of a simple representation of the
quiver $Q$ if and only if 
\begin{itemize}
    \item $\bd \in I$ and the vertex $i$ has no loops or
    \item $\supp(\bd)$ is the extended Dynkin diagram $\tilde{A}_{n-1}$ and $\bd_{|\supp(\bd)} = (1, ... , 1)$  or
    \item $\supp(\bd)$ is a noncyclic strongly connected subquiver and $\langle \bd, \bs_i \rangle \leq 0$ and $\langle \bs_i, \bd \rangle \leq 0$ for all $1 \leq i \leq n$.
\end{itemize}
\end{theorem}

 \begin{remark}
 \begin{enumerate}
     \item Noncyclic here means that $\supp(\bd)$ is not the extended Dynkin diagram $\tilde{A}_{n-1}$, so $Q$ has at least $n+1$ arrows and strongly connected means for every two vertices $i$ and $j$ there is a directed path from $i$ to $j$ and $j$ to $i$.
     \item Note that LeBruyn and Procesi \cite{simplerepresentation} omit the dimension vectors $\bd \in I$, where $i$ is a vertex without loop in their theorem,  but they are essential to obtain a complete list of all dimension vectors of simple representations. Moreover, LeBruyn and Procesi use the notation of $\tilde{A}_n$ for the extended Dynkin diagram with $n$ vertices. This has been exchanged to $\tilde{A}_{n-1}$ to align with the standard notation of extended Dynkin diagrams.  
 \end{enumerate}
 \end{remark}

\begin{definition}
    We call the dimension vectors $\bd$ such that $\bd \in I$ and the vertex $i$ has no loops real simples and all the other dimension vectors of a simple representation imaginary simples. 
\end{definition}

\begin{lemma}
    The dimension vectors given rise to a generator in the construction of the nilpotent Kac-Moody algebra are exactly the real simples and all multiples of the imaginary simples.  
\end{lemma}
\begin{proof}
For a real simple $\bd$, $e_{\bd}$ is a generator of $\mathfrak{g}_Q^{nil}$. Lets assume $\bd$ is a multiple of a imaginary simple. This means, that either $\supp(\bd)$ is the extended Dynkin diagram $\tilde{A}_m$ for some $m$ and $d_{|\supp(\bd)} =(k, \dots k)$ for $k \in \mathbb{N}$ or $supp(\bd)$ is a noncyclic strongly connected subquiver and $\langle \bd, \bs_i \rangle \leq 0$ and $\langle \bs_i, \bd \rangle \leq 0$ for all $s_i \in I$. (For a dimension vector $\bd$ such that $\supp(\bd)$ is noncyclic strongly connected the same is true for $k\bd$ with $k \in \mathbb{N}$. Furthermore, if $\langle \bs_i, \bd \rangle \leq 0$ and $\langle \bd, \bs_i \rangle \leq 0$ for all vertices $i$, we have $\langle \bs_i, k\bd \rangle = k \langle \bs_i, \bd \rangle \leq 0$ and $\langle k\bd, \bs_i \rangle = k \langle \bd, \bs_i \rangle\leq 0$ for all vertices $i$.) If $\bd$ is of the second form, then $e_{\bd}$ is obviously a generator of $\mathfrak{g}_Q^{nil}$. Moreover, if $\bd$ is of the first form we have $\langle s_i, d\rangle = \langle d, s_i \rangle =0$, so $e_{\bd}$ is also a generator of $\mathfrak{g}_Q^{nil}$. \\
Now we assume that $e_{\bd}$ is a generator of $\mathfrak{g}_Q^{nil}$, such that $\bd$ is not a real simple. Then $\supp(\bd)$ is strongly connected and $\langle \bd, \bs_i \rangle \leq 0$ and $\langle \bs_i, \bd \rangle \leq 0$ for all $\bs_i \in I$. If $\supp(\bd)$ is moreover also noncyclic strongly connected, then it is an imaginary simple. Otherwise, $\supp(\bd)$ is the extended Dynkin quiver $\tilde{A}_m$ for some $m$. For the quiver $\tilde{A}_m$ the only dimension vectors $\bd$ with $\langle s_i, d\rangle \leq 0$ and $\langle d, s_i \rangle  \leq 0$ for all $\bs_i \in I$ are the ones of the form $(k, \dots, k)$ for $k \in \mathbb{N}$, so exactly the multiples of the dimension vectors of a simple representation corresponding to $\tilde{A}_n$.
\end{proof}

\section{Interpretation in the Context of BPS Theory}
The following observations are due to B. Davison and L. Hennecart (private communication):\\
In their paper \cite{BPS} Davison, Hennecart and Schlegel Mejia construct in Corollary 1.8. a PBW isomorphism, that relates a certain generalized Kac-Moody Lie algebra $\mathfrak{n}^{BPS, \mathcal{SN},+}_A$ to the Borel-Moore homology of the stack of seminilpotent representations of the preprojective algebra $\Pi_Q$. Since $\mathfrak{n}^{BPS, \mathcal{SN},+}_A$ is (the positive part of) a generalized Kac-Moody Lie algebra, it is entirely determined by its generators, which are given by the $i^!$-restriction of a certain perverse sheaf $\mathcal{G}$ for the inclusion $i:\mathcal{M}_Q \hookrightarrow \mathcal{M}_{\Pi_Q}$ of affine coarse moduli spaces. Using a result of \cite{pointsnilpotentvarities} this isomorphism yields that the generating function of the $d$th graded part of $\mathfrak{n}^{BPS, \mathcal{SN},+}_A$ is the nilpotent Kac polynomial $A_d^{nil}(Q)$ and the constant term is obtained by restricting to the zeroth cohomology. Altogether, Conjecture \ref{conjecture} is equivalent to:

\begin{conjecture}
    The dimension of the derived global sections of $H^0(\mathcal{M}_Q, i^!\mathcal{G}_{\bd})$ is $1$ if $\bd$ satisfies one of the conditions in Definition \ref{def nilp Kac-moody} and $0$ otherwise.
\end{conjecture}

\section{Appendix}
The following tables contain the values of the constant term of the nilpotent Kac polynomial for selected dimension vectors $\bd$ / the dimension of the nilpotent Kac-Moody algebra in degree $\bd$ of the tested quivers listed above. The tables are sorted in ascending order by the number of vertices.
\subsection*{Quivers with $1$ Vertex}

\begin{longtable}{|c|c|c|c|c|c|c|}
\caption{one-vertex Quivers}\label{tab:1vtx}\\
\hline
$\bd$ & $[0]$ & $[1]$ & $[2]$ & $[3]$ & $[4]$ & $[5]$ \\ \hline\hline
\endfirsthead
\multicolumn{7}{l}{\small\itshape Table~\ref{tab:1vtx} continued from previous page} \\ \hline
$\bd$ & $[0]$ & $[1]$ & $[2]$ & $[3]$ & $[4]$ & $[5]$ \\ \hline\hline
\endhead
\hline
\endlastfoot
$(0)$ & 0 & 0 & 0 & 0 & 0 & 0 \\ \hline
$(1)$ & 1 & 1 & 1 & 1 & 1 & 1 \\ \hline
$(2)$ & 0 & 1 & 1 & 1 & 1 & 1 \\ \hline
$(3)$ & 0 & 1 & 2 & 2 & 2 & 2 \\ \hline
$(4)$ & 0 & 1 & 3 & 3 & 3 & 3 \\ \hline
$(5)$ & 0 & 1 & 6 & 6 & 6 & 6 \\ \hline
$(6)$ & 0 & 1 & 9 & 9 & 9 & 9 \\ \hline
$(7)$ & 0 & 1 & 18 & 18 & 18 & 18 \\ \hline
$(8)$ & 0 & 1 & 30 & 30 & 30 & 30 \\ \hline
$(9)$ & 0 & 1 & 56 & 56 & 56 & 56 \\ \hline
$(10)$ & 0 & 1 & 99 & 99 & 99 & 99 \\ \hline
\end{longtable}

\subsection*{Quivers with $2$ Vertices}

\noindent Quivers (adjacency matrices):\\[3pt]
\renewcommand{\arraystretch}{1.4}
\begin{tabular}{lllll}
$Q_{1}=\left(\begin{smallmatrix}1&1\\1&0\end{smallmatrix}\right)$ & $Q_{2}=\left(\begin{smallmatrix}0&2\\2&0\end{smallmatrix}\right)$ & $Q_{3}=\left(\begin{smallmatrix}1&2\\1&0\end{smallmatrix}\right)$ & $Q_{4}=\left(\begin{smallmatrix}3&2\\1&0\end{smallmatrix}\right)$ & $Q_{5}=\left(\begin{smallmatrix}1&2\\2&0\end{smallmatrix}\right)$ \\[5pt]
$Q_{6}=\left(\begin{smallmatrix}1&2\\3&0\end{smallmatrix}\right)$ & $Q_{7}=\left(\begin{smallmatrix}0&2\\1&0\end{smallmatrix}\right)$ & $Q_{8}=\left(\begin{smallmatrix}0&1\\1&0\end{smallmatrix}\right)$ & $Q_{9}=\left(\begin{smallmatrix}0&1\\0&0\end{smallmatrix}\right)$ & $Q_{10}=\left(\begin{smallmatrix}0&3\\1&0\end{smallmatrix}\right)$ \\[5pt]
$Q_{11}=\left(\begin{smallmatrix}1&2\\2&1\end{smallmatrix}\right)$ & $Q_{12}=\left(\begin{smallmatrix}0&3\\2&0\end{smallmatrix}\right)$ & $Q_{13}=\left(\begin{smallmatrix}0&3\\3&0\end{smallmatrix}\right)$ & $Q_{14}=\left(\begin{smallmatrix}1&1\\0&1\end{smallmatrix}\right)$ & $Q_{15}=\left(\begin{smallmatrix}1&1\\1&1\end{smallmatrix}\right)$ \\[5pt]
$Q_{16}=\left(\begin{smallmatrix}2&1\\1&2\end{smallmatrix}\right)$ & $Q_{17}=\left(\begin{smallmatrix}1&1\\0&0\end{smallmatrix}\right)$ & $Q_{18}=\left(\begin{smallmatrix}1&0\\0&1\end{smallmatrix}\right)$ & $Q_{19}=\left(\begin{smallmatrix}2&1\\1&1\end{smallmatrix}\right)$ &  \\[5pt]
\end{tabular}
\renewcommand{\arraystretch}{1.0}
 
\small
\setlength{\tabcolsep}{1.5pt}
\begin{longtable}{|c|c|c|c|c|c|c|c|c|c|c|c|c|c|c|c|c|c|c|c|}
\caption{two-vertex Quivers}\label{tab:2vtx}\\
\hline
$\bd$ & $Q_{1}$ & $Q_{2}$ & $Q_{3}$ & $Q_{4}$ & $Q_{5}$ & $Q_{6}$ & $Q_{7}$ & $Q_{8}$ & $Q_{9}$ & $Q_{10}$ & $Q_{11}$ & $Q_{12}$ & $Q_{13}$ & $Q_{14}$ & $Q_{15}$ & $Q_{16}$ & $Q_{17}$ & $Q_{18}$ & $Q_{19}$ \\ \hline\hline
\endfirsthead
\multicolumn{20}{l}{\small\itshape Table~\ref{tab:2vtx} continued from previous page} \\ \hline
$\bd$ & $Q_{1}$ & $Q_{2}$ & $Q_{3}$ & $Q_{4}$ & $Q_{5}$ & $Q_{6}$ & $Q_{7}$ & $Q_{8}$ & $Q_{9}$ & $Q_{10}$ & $Q_{11}$ & $Q_{12}$ & $Q_{13}$ & $Q_{14}$ & $Q_{15}$ & $Q_{16}$ & $Q_{17}$ & $Q_{18}$ & $Q_{19}$ \\ \hline\hline
\endhead
\hline \multicolumn{20}{r}{\small\itshape continued on next page} \\
\endfoot
\hline
\endlastfoot
\multicolumn{20}{|l|}{\textbf{total degree }$|\bd|=0$} \\ \hline
$(0,0)$ & 0 & 0 & 0 & 0 & 0 & 0 & 0 & 0 & 0 & 0 & 0 & 0 & 0 & 0 & 0 & 0 & 0 & 0 & 0 \\ \hline
\multicolumn{20}{|l|}{\textbf{total degree }$|\bd|=1$} \\ \hline
$(1,0)$ & 1 & 1 & 1 & 1 & 1 & 1 & 1 & 1 & 1 & 1 & 1 & 1 & 1 & 1 & 1 & 1 & 1 & 1 & 1 \\ \hline
$(0,1)$ & 1 & 1 & 1 & 1 & 1 & 1 & 1 & 1 & 1 & 1 & 1 & 1 & 1 & 1 & 1 & 1 & 1 & 1 & 1 \\ \hline
\multicolumn{20}{|l|}{\textbf{total degree }$|\bd|=2$} \\ \hline
$(2,0)$ & 1 & 0 & 1 & 1 & 1 & 1 & 0 & 0 & 0 & 0 & 1 & 0 & 0 & 1 & 1 & 1 & 1 & 1 & 1 \\ \hline
$(1,1)$ & 2 & 2 & 2 & 2 & 2 & 2 & 2 & 2 & 1 & 2 & 2 & 2 & 2 & 1 & 2 & 2 & 1 & 0 & 2 \\ \hline
$(0,2)$ & 0 & 0 & 0 & 0 & 0 & 0 & 0 & 0 & 0 & 0 & 1 & 0 & 0 & 1 & 1 & 1 & 0 & 1 & 1 \\ \hline
\multicolumn{20}{|l|}{\textbf{total degree }$|\bd|=3$} \\ \hline
$(3,0)$ & 1 & 0 & 1 & 2 & 1 & 1 & 0 & 0 & 0 & 0 & 1 & 0 & 0 & 1 & 1 & 2 & 1 & 1 & 2 \\ \hline
$(2,1)$ & 4 & 3 & 4 & 4 & 4 & 4 & 2 & 1 & 0 & 2 & 4 & 3 & 3 & 2 & 4 & 4 & 2 & 0 & 4 \\ \hline
$(1,2)$ & 1 & 3 & 2 & 2 & 3 & 3 & 2 & 1 & 0 & 2 & 4 & 3 & 3 & 2 & 4 & 4 & 0 & 0 & 4 \\ \hline
$(0,3)$ & 0 & 0 & 0 & 0 & 0 & 0 & 0 & 0 & 0 & 0 & 1 & 0 & 0 & 1 & 1 & 2 & 0 & 1 & 1 \\ \hline
\multicolumn{20}{|l|}{\textbf{total degree }$|\bd|=4$} \\ \hline
$(4,0)$ & 1 & 0 & 1 & 3 & 1 & 1 & 0 & 0 & 0 & 0 & 1 & 0 & 0 & 1 & 1 & 3 & 1 & 1 & 3 \\ \hline
$(3,1)$ & 7 & 2 & 7 & 8 & 7 & 7 & 1 & 0 & 0 & 2 & 7 & 3 & 4 & 3 & 7 & 8 & 3 & 0 & 8 \\ \hline
$(2,2)$ & 5 & 6 & 6 & 6 & 7 & 7 & 4 & 2 & 0 & 4 & 9 & 6 & 6 & 4 & 9 & 9 & 1 & 0 & 9 \\ \hline
$(1,3)$ & 0 & 2 & 1 & 1 & 2 & 3 & 1 & 0 & 0 & 2 & 7 & 3 & 4 & 3 & 7 & 8 & 0 & 0 & 7 \\ \hline
$(0,4)$ & 0 & 0 & 0 & 0 & 0 & 0 & 0 & 0 & 0 & 0 & 1 & 0 & 0 & 1 & 1 & 3 & 0 & 1 & 1 \\ \hline
\multicolumn{20}{|l|}{\textbf{total degree }$|\bd|=5$} \\ \hline
$(5,0)$ & 1 & 0 & 1 & 6 & 1 & 1 & 0 & 0 & 0 & 0 & 1 & 0 & 0 & 1 & 1 & 6 & 1 & 1 & 6 \\ \hline
$(4,1)$ & 12 & 1 & 12 & 16 & 12 & 12 & 0 & 0 & 0 & 1 & 12 & 2 & 3 & 5 & 12 & 16 & 5 & 0 & 16 \\ \hline
$(3,2)$ & 15 & 11 & 17 & 18 & 19 & 19 & 6 & 1 & 0 & 7 & 22 & 12 & 13 & 8 & 22 & 24 & 3 & 0 & 24 \\ \hline
$(2,3)$ & 4 & 11 & 8 & 8 & 12 & 13 & 6 & 1 & 0 & 7 & 22 & 12 & 13 & 8 & 22 & 24 & 0 & 0 & 22 \\ \hline
$(1,4)$ & 0 & 1 & 0 & 0 & 1 & 2 & 0 & 0 & 0 & 1 & 12 & 2 & 3 & 5 & 12 & 16 & 0 & 0 & 12 \\ \hline
$(0,5)$ & 0 & 0 & 0 & 0 & 0 & 0 & 0 & 0 & 0 & 0 & 1 & 0 & 0 & 1 & 1 & 6 & 0 & 1 & 1 \\ \hline
\multicolumn{20}{|l|}{\textbf{total degree }$|\bd|=6$} \\ \hline
$(5,1)$ & 19 & 0 & 19 & 32 & 19 & 19 & 0 & 0 & 0 & 0 & 19 & 1 & 2 & 7 & 19 & 32 & 7 & 0 & 32 \\ \hline
$(4,2)$ & 33 & 12 & 36 & 42 & 39 & 39 & 4 & 0 & 0 & 7 & 44 & 15 & 18 & 14 & 44 & 54 & 6 & 0 & 54 \\ \hline
$(3,3)$ & 19 & 27 & 28 & 29 & 37 & 39 & 14 & 2 & 0 & 16 & 59 & 29 & 31 & 17 & 59 & 66 & 1 & 0 & 62 \\ \hline
$(2,4)$ & 1 & 12 & 6 & 6 & 13 & 16 & 4 & 0 & 0 & 7 & 44 & 15 & 18 & 14 & 44 & 54 & 0 & 0 & 44 \\ \hline
$(1,5)$ & 0 & 0 & 0 & 0 & 0 & 1 & 0 & 0 & 0 & 0 & 19 & 1 & 2 & 7 & 19 & 32 & 0 & 0 & 19 \\ \hline
\multicolumn{20}{|l|}{\textbf{total degree }$|\bd|=7$} \\ \hline
$(5,2)$ & 73 & 11 & 78 & 104 & 83 & 83 & 2 & 0 & 0 & 7 & 90 & 18 & 25 & 25 & 90 & 128 & 13 & 0 & 128 \\ \hline
$(4,3)$ & 63 & 50 & 83 & 90 & 103 & 106 & 21 & 1 & 0 & 27 & 147 & 57 & 64 & 37 & 147 & 176 & 5 & 0 & 168 \\ \hline
$(3,4)$ & 15 & 50 & 36 & 37 & 61 & 69 & 21 & 1 & 0 & 27 & 147 & 57 & 64 & 37 & 147 & 176 & 0 & 0 & 152 \\ \hline
$(2,5)$ & 0 & 11 & 4 & 4 & 12 & 19 & 2 & 0 & 0 & 7 & 90 & 18 & 25 & 25 & 90 & 128 & 0 & 0 & 90 \\ \hline
\multicolumn{20}{|l|}{\textbf{total degree }$|\bd|=8$} \\ \hline
$(5,3)$ & 171 & 70 & 212 & 252 & 253 & 258 & 21 & 0 & 0 & 36 & 337 & 90 & 110 & 72 & 337 & 448 & 13 & 0 & 432 \\ \hline
$(4,4)$ & 74 & 124 & 135 & 142 & 204 & 222 & 48 & 2 & 0 & 61 & 412 & 141 & 160 & 85 & 412 & 511 & 1 & 0 & 455 \\ \hline
$(3,5)$ & 7 & 70 & 36 & 37 & 81 & 103 & 21 & 0 & 0 & 36 & 337 & 90 & 110 & 72 & 337 & 448 & 0 & 0 & 344 \\ \hline
\multicolumn{20}{|l|}{\textbf{total degree }$|\bd|=9$} \\ \hline
$(5,4)$ & 277 & 248 & 437 & 488 & 612 & 650 & 80 & 1 & 0 & 113 & 1087 & 297 & 350 & 189 & 1087 & 1440 & 7 & 0 & 1312 \\ \hline
$(4,5)$ & 63 & 248 & 184 & 191 & 350 & 413 & 80 & 1 & 0 & 113 & 1087 & 297 & 350 & 189 & 1087 & 1440 & 0 & 0 & 1168 \\ \hline
\multicolumn{20}{|l|}{\textbf{total degree }$|\bd|=10$} \\ \hline
$(5,5)$ & 325 & 636 & 725 & 782 & 1238 & 1398 & 190 & 2 & 0 & 266 & 3175 & 766 & 912 & 454 & 3175 & 4358 & 1 & 0 & 3670 \\ \hline
\end{longtable}

\subsection*{Quivers with 3 Vertices}

\noindent Quivers (adjacency matrices):\\[4pt]
\renewcommand{\arraystretch}{1.6}
\begin{tabular}{llll}
$Q_{1}=\left(\begin{smallmatrix}1&1&0\\0&1&1\\0&0&1\end{smallmatrix}\right)$ & $Q_{2}=\left(\begin{smallmatrix}1&1&0\\0&1&1\\0&1&1\end{smallmatrix}\right)$ & $Q_{3}=\left(\begin{smallmatrix}1&1&0\\1&1&1\\0&1&1\end{smallmatrix}\right)$ & $Q_{4}=\left(\begin{smallmatrix}1&1&1\\1&1&1\\0&1&1\end{smallmatrix}\right)$ \\[4pt]
$Q_{5}=\left(\begin{smallmatrix}1&1&1\\1&1&1\\1&1&1\end{smallmatrix}\right)$ & $Q_{6}=\left(\begin{smallmatrix}2&1&0\\0&2&1\\0&0&2\end{smallmatrix}\right)$ & $Q_{7}=\left(\begin{smallmatrix}2&1&0\\0&2&1\\0&1&2\end{smallmatrix}\right)$ & $Q_{8}=\left(\begin{smallmatrix}2&1&0\\1&2&1\\0&1&2\end{smallmatrix}\right)$ \\[4pt]
$Q_{9}=\left(\begin{smallmatrix}2&1&1\\1&2&1\\1&1&2\end{smallmatrix}\right)$ & $Q_{10}=\left(\begin{smallmatrix}2&1&0\\1&1&1\\0&1&2\end{smallmatrix}\right)$ & $Q_{11}=\left(\begin{smallmatrix}2&1&0\\1&2&1\\0&1&1\end{smallmatrix}\right)$ & $Q_{12}=\left(\begin{smallmatrix}1&1&0\\0&2&1\\0&1&1\end{smallmatrix}\right)$ \\[4pt]
$Q_{13}=\left(\begin{smallmatrix}1&1&0\\0&1&1\\0&1&2\end{smallmatrix}\right)$ & $Q_{14}=\left(\begin{smallmatrix}1&1&0\\1&2&1\\0&1&1\end{smallmatrix}\right)$ &  &  \\[4pt]
\end{tabular}
\renewcommand{\arraystretch}{1.0}

\small
\setlength{\tabcolsep}{3pt}
\begin{longtable}{|l|l|l|l|l|l|l|l|l|l|l|l|l|l|l|}
\caption{three-vertex Quivers}\label{tab:3vtx}\\
\hline
$\bd$ & $Q_{1}$ & $Q_{2}$ & $Q_{3}$ & $Q_{4}$ & $Q_{5}$ & $Q_{6}$ & $Q_{7}$ & $Q_{8}$ & $Q_{9}$ & $Q_{10}$ & $Q_{11}$ & $Q_{12}$ & $Q_{13}$ & $Q_{14}$ \\ \hline\hline
\endfirsthead
\multicolumn{15}{l}{\small\itshape Table~\ref{tab:3vtx} continued from previous page} \\ \hline
$\bd$ & $Q_{1}$ & $Q_{2}$ & $Q_{3}$ & $Q_{4}$ & $Q_{5}$ & $Q_{6}$ & $Q_{7}$ & $Q_{8}$ & $Q_{9}$ & $Q_{10}$ & $Q_{11}$ & $Q_{12}$ & $Q_{13}$ & $Q_{14}$ \\ \hline\hline
\endhead
\hline \multicolumn{15}{r}{\small\itshape continued on next page} \\
\endfoot
\hline
\endlastfoot
\multicolumn{15}{|l|}{\textbf{total degree }$|\bd|=0$} \\ \hline
$(0,0,0)$ & 0 & 0 & 0 & 0 & 0 & 0 & 0 & 0 & 0 & 0 & 0 & 0 & 0 & 0 \\ \hline
\multicolumn{15}{|l|}{\textbf{total degree }$|\bd|=1$} \\ \hline
$(1,0,0)$ & 1 & 1 & 1 & 1 & 1 & 1 & 1 & 1 & 1 & 1 & 1 & 1 & 1 & 1 \\ \hline
$(0,1,0)$ & 1 & 1 & 1 & 1 & 1 & 1 & 1 & 1 & 1 & 1 & 1 & 1 & 1 & 1 \\ \hline
$(0,0,1)$ & 1 & 1 & 1 & 1 & 1 & 1 & 1 & 1 & 1 & 1 & 1 & 1 & 1 & 1 \\ \hline
\multicolumn{15}{|l|}{\textbf{total degree }$|\bd|=2$} \\ \hline
$(2,0,0)$ & 1 & 1 & 1 & 1 & 1 & 1 & 1 & 1 & 1 & 1 & 1 & 1 & 1 & 1 \\ \hline
$(1,1,0)$ & 1 & 1 & 2 & 2 & 2 & 1 & 1 & 2 & 2 & 2 & 2 & 1 & 1 & 2 \\ \hline
$(1,0,1)$ & 0 & 0 & 0 & 1 & 2 & 0 & 0 & 0 & 2 & 0 & 0 & 0 & 0 & 0 \\ \hline
$(0,2,0)$ & 1 & 1 & 1 & 1 & 1 & 1 & 1 & 1 & 1 & 1 & 1 & 1 & 1 & 1 \\ \hline
$(0,1,1)$ & 1 & 2 & 2 & 2 & 2 & 1 & 2 & 2 & 2 & 2 & 2 & 2 & 2 & 2 \\ \hline
$(0,0,2)$ & 1 & 1 & 1 & 1 & 1 & 1 & 1 & 1 & 1 & 1 & 1 & 1 & 1 & 1 \\ \hline
\multicolumn{15}{|l|}{\textbf{total degree }$|\bd|=3$} \\ \hline
$(3,0,0)$ & 1 & 1 & 1 & 1 & 1 & 2 & 2 & 2 & 2 & 2 & 2 & 1 & 1 & 1 \\ \hline
$(2,1,0)$ & 2 & 2 & 4 & 4 & 4 & 2 & 2 & 4 & 4 & 4 & 4 & 2 & 2 & 4 \\ \hline
$(2,0,1)$ & 0 & 0 & 0 & 2 & 4 & 0 & 0 & 0 & 4 & 0 & 0 & 0 & 0 & 0 \\ \hline
$(1,2,0)$ & 2 & 2 & 4 & 4 & 4 & 2 & 2 & 4 & 4 & 4 & 4 & 2 & 2 & 4 \\ \hline
$(1,1,1)$ & 1 & 2 & 4 & 5 & 6 & 1 & 2 & 4 & 6 & 4 & 4 & 2 & 2 & 4 \\ \hline
$(1,0,2)$ & 0 & 0 & 0 & 2 & 4 & 0 & 0 & 0 & 4 & 0 & 0 & 0 & 0 & 0 \\ \hline
$(0,3,0)$ & 1 & 1 & 1 & 1 & 1 & 2 & 2 & 2 & 2 & 1 & 2 & 2 & 1 & 2 \\ \hline
$(0,2,1)$ & 2 & 4 & 4 & 4 & 4 & 2 & 4 & 4 & 4 & 4 & 4 & 4 & 4 & 4 \\ \hline
$(0,1,2)$ & 2 & 4 & 4 & 4 & 4 & 2 & 4 & 4 & 4 & 4 & 4 & 4 & 4 & 4 \\ \hline
$(0,0,3)$ & 1 & 1 & 1 & 1 & 1 & 2 & 2 & 2 & 2 & 2 & 1 & 1 & 2 & 1 \\ \hline
\multicolumn{15}{|l|}{\textbf{total degree }$|\bd|=4$} \\ \hline
$(3,1,0)$ & 3 & 3 & 7 & 7 & 7 & 4 & 4 & 8 & 8 & 8 & 8 & 3 & 3 & 7 \\ \hline
$(3,0,1)$ & 0 & 0 & 0 & 3 & 7 & 0 & 0 & 0 & 8 & 0 & 0 & 0 & 0 & 0 \\ \hline
$(2,2,0)$ & 4 & 4 & 9 & 9 & 9 & 4 & 4 & 9 & 9 & 9 & 9 & 4 & 4 & 9 \\ \hline
$(2,1,1)$ & 2 & 4 & 8 & 12 & 16 & 2 & 4 & 8 & 16 & 8 & 8 & 4 & 4 & 8 \\ \hline
$(2,0,2)$ & 0 & 0 & 0 & 4 & 9 & 0 & 0 & 0 & 9 & 0 & 0 & 0 & 0 & 0 \\ \hline
$(1,3,0)$ & 3 & 3 & 7 & 7 & 7 & 4 & 4 & 8 & 8 & 7 & 8 & 4 & 3 & 8 \\ \hline
$(1,2,1)$ & 3 & 6 & 12 & 14 & 16 & 3 & 6 & 12 & 16 & 12 & 12 & 6 & 6 & 12 \\ \hline
$(1,1,2)$ & 2 & 4 & 8 & 12 & 16 & 2 & 4 & 8 & 16 & 8 & 8 & 4 & 4 & 8 \\ \hline
$(1,0,3)$ & 0 & 0 & 0 & 3 & 7 & 0 & 0 & 0 & 8 & 0 & 0 & 0 & 0 & 0 \\ \hline
$(0,3,1)$ & 3 & 7 & 7 & 7 & 7 & 4 & 8 & 8 & 8 & 7 & 8 & 8 & 7 & 8 \\ \hline
$(0,2,2)$ & 4 & 9 & 9 & 9 & 9 & 4 & 9 & 9 & 9 & 9 & 9 & 9 & 9 & 9 \\ \hline
$(0,1,3)$ & 3 & 7 & 7 & 7 & 7 & 4 & 8 & 8 & 8 & 8 & 7 & 7 & 8 & 7 \\ \hline
\multicolumn{15}{|l|}{\textbf{total degree }$|\bd|=5$} \\ \hline
$(3,2,0)$ & 8 & 8 & 22 & 22 & 22 & 10 & 10 & 24 & 24 & 24 & 24 & 8 & 8 & 22 \\ \hline
$(3,1,1)$ & 3 & 6 & 14 & 25 & 37 & 4 & 8 & 16 & 40 & 16 & 16 & 6 & 6 & 14 \\ \hline
$(3,0,2)$ & 0 & 0 & 0 & 8 & 22 & 0 & 0 & 0 & 24 & 0 & 0 & 0 & 0 & 0 \\ \hline
$(2,3,0)$ & 8 & 8 & 22 & 22 & 22 & 10 & 10 & 24 & 24 & 22 & 24 & 10 & 8 & 24 \\ \hline
$(2,2,1)$ & 7 & 14 & 32 & 42 & 52 & 7 & 14 & 32 & 52 & 32 & 32 & 14 & 14 & 32 \\ \hline
$(2,1,2)$ & 4 & 8 & 16 & 33 & 52 & 4 & 8 & 16 & 52 & 16 & 16 & 8 & 8 & 16 \\ \hline
$(2,0,3)$ & 0 & 0 & 0 & 8 & 22 & 0 & 0 & 0 & 24 & 0 & 0 & 0 & 0 & 0 \\ \hline
$(1,3,1)$ & 7 & 15 & 31 & 34 & 37 & 8 & 16 & 32 & 40 & 31 & 32 & 16 & 15 & 32 \\ \hline
$(1,2,2)$ & 7 & 16 & 32 & 42 & 52 & 7 & 16 & 32 & 52 & 32 & 32 & 16 & 16 & 32 \\ \hline
$(1,1,3)$ & 3 & 7 & 14 & 25 & 37 & 4 & 8 & 16 & 40 & 16 & 14 & 7 & 8 & 14 \\ \hline
$(0,3,2)$ & 8 & 22 & 22 & 22 & 22 & 10 & 24 & 24 & 24 & 22 & 24 & 24 & 22 & 24 \\ \hline
$(0,2,3)$ & 8 & 22 & 22 & 22 & 22 & 10 & 24 & 24 & 24 & 24 & 22 & 22 & 24 & 22 \\ \hline
\multicolumn{15}{|l|}{\textbf{total degree }$|\bd|=6$} \\ \hline
$(3,3,0)$ & 17 & 17 & 59 & 59 & 59 & 24 & 24 & 66 & 66 & 62 & 66 & 20 & 17 & 62 \\ \hline
$(3,2,1)$ & 13 & 26 & 74 & 108 & 144 & 16 & 32 & 80 & 152 & 80 & 80 & 26 & 26 & 74 \\ \hline
$(3,1,2)$ & 6 & 12 & 28 & 78 & 144 & 8 & 16 & 32 & 152 & 32 & 32 & 12 & 12 & 28 \\ \hline
$(3,0,3)$ & 0 & 0 & 0 & 17 & 59 & 0 & 0 & 0 & 66 & 0 & 0 & 0 & 0 & 0 \\ \hline
$(2,3,1)$ & 19 & 40 & 102 & 123 & 144 & 21 & 42 & 104 & 152 & 102 & 104 & 42 & 40 & 104 \\ \hline
$(2,2,2)$ & 16 & 37 & 86 & 140 & 199 & 16 & 37 & 86 & 199 & 86 & 86 & 37 & 37 & 86 \\ \hline
$(2,1,3)$ & 6 & 14 & 28 & 78 & 144 & 8 & 16 & 32 & 152 & 32 & 28 & 14 & 16 & 28 \\ \hline
$(1,3,2)$ & 19 & 50 & 102 & 123 & 144 & 21 & 52 & 104 & 152 & 102 & 104 & 52 & 50 & 104 \\ \hline
$(1,2,3)$ & 13 & 37 & 74 & 108 & 144 & 16 & 40 & 80 & 152 & 80 & 74 & 37 & 40 & 74 \\ \hline
$(0,3,3)$ & 17 & 59 & 59 & 59 & 59 & 24 & 66 & 66 & 66 & 62 & 62 & 62 & 62 & 62 \\ \hline
\multicolumn{15}{|l|}{\textbf{total degree }$|\bd|=7$} \\ \hline
$(3,3,1)$ & 42 & 87 & 285 & 374 & 466 & 53 & 106 & 304 & 504 & 300 & 304 & 90 & 87 & 288 \\ \hline
$(3,2,2)$ & 31 & 72 & 208 & 408 & 654 & 38 & 88 & 224 & 680 & 224 & 224 & 72 & 72 & 208 \\ \hline
$(3,1,3)$ & 9 & 21 & 49 & 207 & 466 & 16 & 32 & 64 & 504 & 64 & 56 & 21 & 24 & 49 \\ \hline
$(2,3,2)$ & 52 & 136 & 348 & 496 & 654 & 56 & 140 & 352 & 680 & 348 & 352 & 140 & 136 & 352 \\ \hline
$(2,2,3)$ & 31 & 89 & 208 & 408 & 654 & 38 & 96 & 224 & 680 & 224 & 208 & 89 & 96 & 208 \\ \hline
$(1,3,3)$ & 42 & 141 & 285 & 374 & 466 & 53 & 152 & 304 & 504 & 300 & 288 & 144 & 148 & 288 \\ \hline
\multicolumn{15}{|l|}{\textbf{total degree }$|\bd|=8$} \\ \hline
$(3,3,2)$ & 116 & 302 & 998 & 1654 & 2421 & 143 & 360 & 1056 & 2568 & 1048 & 1056 & 308 & 302 & 1004 \\ \hline
$(3,2,3)$ & 59 & 171 & 499 & 1284 & 2421 & 88 & 224 & 576 & 2568 & 576 & 536 & 171 & 184 & 499 \\ \hline
$(2,3,3)$ & 116 & 389 & 998 & 1654 & 2421 & 143 & 416 & 1056 & 2568 & 1048 & 1004 & 395 & 408 & 1004 \\ \hline
\multicolumn{15}{|l|}{\textbf{total degree }$|\bd|=9$} \\ \hline
$(3,3,3)$ & 261 & 876 & 2915 & 5976 & 10068 & 368 & 1082 & 3220 & 10790 & 3204 & 3068 & 885 & 918 & 2924 \\ \hline
\end{longtable}

\subsection*{Quiver with $4$ Vertices}

\noindent Quivers (adjacency matrices):\\[4pt]
\renewcommand{\arraystretch}{1.5}
\begin{tabular}{lll}
$Q_{1}=\left(\begin{smallmatrix}1&1&0&0\\1&1&1&0\\0&1&1&1\\0&0&1&1\end{smallmatrix}\right)$ & $Q_{2}=\left(\begin{smallmatrix}0&1&0&0\\1&0&1&0\\0&1&0&1\\0&0&1&0\end{smallmatrix}\right)$ & $Q_{3}=\left(\begin{smallmatrix}0&1&0&0\\1&0&1&0\\0&0&0&1\\0&0&1&0\end{smallmatrix}\right)$ \\[6pt]
$Q_{4}=\left(\begin{smallmatrix}0&1&0&0\\1&0&0&0\\0&0&0&1\\0&0&1&0\end{smallmatrix}\right)$ & $Q_{5}=\left(\begin{smallmatrix}1&1&0&0\\1&1&1&0\\0&0&1&1\\0&0&1&1\end{smallmatrix}\right)$ & $Q_{6}=\left(\begin{smallmatrix}1&1&0&0\\1&1&0&0\\0&0&1&1\\0&0&1&1\end{smallmatrix}\right)$ \\[6pt]
$Q_{7}=\left(\begin{smallmatrix}1&1&0&0\\1&1&1&0\\0&1&1&1\\0&0&0&1\end{smallmatrix}\right)$ & $Q_{8}=\left(\begin{smallmatrix}1&1&0&0\\1&1&1&0\\0&0&1&1\\0&0&0&1\end{smallmatrix}\right)$ & $Q_{9}=\left(\begin{smallmatrix}1&1&0&0\\0&1&1&0\\0&0&1&1\\0&0&0&1\end{smallmatrix}\right)$ \\[6pt]
\end{tabular}
\renewcommand{\arraystretch}{1.0}

\small
\setlength{\tabcolsep}{3pt}
\begin{longtable}{|l|l|l|l|l|l|l|l|l|l|}
\caption{four-vertex Quivers}\label{tab:4vtx}\\
\hline
$\bd$ & $Q_{1}$ & $Q_{2}$ & $Q_{3}$ & $Q_{4}$ & $Q_{5}$ & $Q_{6}$ & $Q_{7}$ & $Q_{8}$ & $Q_{9}$ \\ \hline\hline
\endfirsthead
\multicolumn{10}{l}{\small\itshape Table~\ref{tab:4vtx} continued from previous page} \\ \hline
$\bd$ & $Q_{1}$ & $Q_{2}$ & $Q_{3}$ & $Q_{4}$ & $Q_{5}$ & $Q_{6}$ & $Q_{7}$ & $Q_{8}$ & $Q_{9}$ \\ \hline\hline
\endhead
\hline \multicolumn{10}{r}{\small\itshape continued on next page} \\
\endfoot
\hline
\endlastfoot
\multicolumn{10}{|l|}{\textbf{total degree }$|\bd|=0$} \\ \hline
$(0,0,0,0)$ & 0 & 0 & 0 & 0 & 0 & 0 & 0 & 0 & 0 \\ \hline
\multicolumn{10}{|l|}{\textbf{total degree }$|\bd|=1$} \\ \hline
$(1,0,0,0)$ & 1 & 1 & 1 & 1 & 1 & 1 & 1 & 1 & 1 \\ \hline
$(0,1,0,0)$ & 1 & 1 & 1 & 1 & 1 & 1 & 1 & 1 & 1 \\ \hline
$(0,0,1,0)$ & 1 & 1 & 1 & 1 & 1 & 1 & 1 & 1 & 1 \\ \hline
$(0,0,0,1)$ & 1 & 1 & 1 & 1 & 1 & 1 & 1 & 1 & 1 \\ \hline
\multicolumn{10}{|l|}{\textbf{total degree }$|\bd|=2$} \\ \hline
$(2,0,0,0)$ & 1 & 0 & 0 & 0 & 1 & 1 & 1 & 1 & 1 \\ \hline
$(1,1,0,0)$ & 2 & 2 & 2 & 2 & 2 & 2 & 2 & 2 & 1 \\ \hline
$(1,0,1,0)$ & 0 & 0 & 0 & 0 & 0 & 0 & 0 & 0 & 0 \\ \hline
$(1,0,0,1)$ & 0 & 0 & 0 & 0 & 0 & 0 & 0 & 0 & 0 \\ \hline
$(0,2,0,0)$ & 1 & 0 & 0 & 0 & 1 & 1 & 1 & 1 & 1 \\ \hline
$(0,1,1,0)$ & 2 & 2 & 1 & 0 & 1 & 0 & 2 & 1 & 1 \\ \hline
$(0,1,0,1)$ & 0 & 0 & 0 & 0 & 0 & 0 & 0 & 0 & 0 \\ \hline
$(0,0,2,0)$ & 1 & 0 & 0 & 0 & 1 & 1 & 1 & 1 & 1 \\ \hline
$(0,0,1,1)$ & 2 & 2 & 2 & 2 & 2 & 2 & 1 & 1 & 1 \\ \hline
$(0,0,0,2)$ & 1 & 0 & 0 & 0 & 1 & 1 & 1 & 1 & 1 \\ \hline
\multicolumn{10}{|l|}{\textbf{total degree }$|\bd|=3$} \\ \hline
$(2,1,0,0)$ & 4 & 1 & 1 & 1 & 4 & 4 & 4 & 4 & 2 \\ \hline
$(2,0,1,0)$ & 0 & 0 & 0 & 0 & 0 & 0 & 0 & 0 & 0 \\ \hline
$(2,0,0,1)$ & 0 & 0 & 0 & 0 & 0 & 0 & 0 & 0 & 0 \\ \hline
$(1,2,0,0)$ & 4 & 1 & 1 & 1 & 4 & 4 & 4 & 4 & 2 \\ \hline
$(1,1,1,0)$ & 4 & 4 & 2 & 0 & 2 & 0 & 4 & 2 & 1 \\ \hline
$(1,1,0,1)$ & 0 & 0 & 0 & 0 & 0 & 0 & 0 & 0 & 0 \\ \hline
$(1,0,2,0)$ & 0 & 0 & 0 & 0 & 0 & 0 & 0 & 0 & 0 \\ \hline
$(1,0,1,1)$ & 0 & 0 & 0 & 0 & 0 & 0 & 0 & 0 & 0 \\ \hline
$(1,0,0,2)$ & 0 & 0 & 0 & 0 & 0 & 0 & 0 & 0 & 0 \\ \hline
$(0,2,1,0)$ & 4 & 1 & 0 & 0 & 2 & 0 & 4 & 2 & 2 \\ \hline
$(0,2,0,1)$ & 0 & 0 & 0 & 0 & 0 & 0 & 0 & 0 & 0 \\ \hline
$(0,1,2,0)$ & 4 & 1 & 0 & 0 & 2 & 0 & 4 & 2 & 2 \\ \hline
$(0,1,1,1)$ & 4 & 4 & 2 & 0 & 2 & 0 & 2 & 1 & 1 \\ \hline
$(0,1,0,2)$ & 0 & 0 & 0 & 0 & 0 & 0 & 0 & 0 & 0 \\ \hline
$(0,0,2,1)$ & 4 & 1 & 1 & 1 & 4 & 4 & 2 & 2 & 2 \\ \hline
$(0,0,1,2)$ & 4 & 1 & 1 & 1 & 4 & 4 & 2 & 2 & 2 \\ \hline
\multicolumn{10}{|l|}{\textbf{total degree }$|\bd|=4$} \\ \hline
$(2,2,0,0)$ & 9 & 2 & 2 & 2 & 9 & 9 & 9 & 9 & 4 \\ \hline
$(2,1,1,0)$ & 8 & 2 & 1 & 0 & 4 & 0 & 8 & 4 & 2 \\ \hline
$(2,1,0,1)$ & 0 & 0 & 0 & 0 & 0 & 0 & 0 & 0 & 0 \\ \hline
$(2,0,2,0)$ & 0 & 0 & 0 & 0 & 0 & 0 & 0 & 0 & 0 \\ \hline
$(2,0,1,1)$ & 0 & 0 & 0 & 0 & 0 & 0 & 0 & 0 & 0 \\ \hline
$(2,0,0,2)$ & 0 & 0 & 0 & 0 & 0 & 0 & 0 & 0 & 0 \\ \hline
$(1,2,1,0)$ & 12 & 7 & 2 & 0 & 6 & 0 & 12 & 6 & 3 \\ \hline
$(1,2,0,1)$ & 0 & 0 & 0 & 0 & 0 & 0 & 0 & 0 & 0 \\ \hline
$(1,1,2,0)$ & 8 & 2 & 0 & 0 & 4 & 0 & 8 & 4 & 2 \\ \hline
$(1,1,1,1)$ & 8 & 8 & 4 & 0 & 4 & 0 & 4 & 2 & 1 \\ \hline
$(1,1,0,2)$ & 0 & 0 & 0 & 0 & 0 & 0 & 0 & 0 & 0 \\ \hline
$(1,0,2,1)$ & 0 & 0 & 0 & 0 & 0 & 0 & 0 & 0 & 0 \\ \hline
$(1,0,1,2)$ & 0 & 0 & 0 & 0 & 0 & 0 & 0 & 0 & 0 \\ \hline
$(0,2,2,0)$ & 9 & 2 & 0 & 0 & 4 & 0 & 9 & 4 & 4 \\ \hline
$(0,2,1,1)$ & 8 & 2 & 0 & 0 & 4 & 0 & 4 & 2 & 2 \\ \hline
$(0,2,0,2)$ & 0 & 0 & 0 & 0 & 0 & 0 & 0 & 0 & 0 \\ \hline
$(0,1,2,1)$ & 12 & 7 & 2 & 0 & 6 & 0 & 6 & 3 & 3 \\ \hline
$(0,1,1,2)$ & 8 & 2 & 1 & 0 & 4 & 0 & 4 & 2 & 2 \\ \hline
$(0,0,2,2)$ & 9 & 2 & 2 & 2 & 9 & 9 & 4 & 4 & 4 \\ \hline
\multicolumn{10}{|l|}{\textbf{total degree }$|\bd|=5$} \\ \hline
$(2,2,1,0)$ & 32 & 13 & 5 & 0 & 16 & 0 & 32 & 16 & 7 \\ \hline
$(2,2,0,1)$ & 0 & 0 & 0 & 0 & 0 & 0 & 0 & 0 & 0 \\ \hline
$(2,1,2,0)$ & 16 & 1 & 0 & 0 & 8 & 0 & 16 & 8 & 4 \\ \hline
$(2,1,1,1)$ & 16 & 4 & 2 & 0 & 8 & 0 & 8 & 4 & 2 \\ \hline
$(2,1,0,2)$ & 0 & 0 & 0 & 0 & 0 & 0 & 0 & 0 & 0 \\ \hline
$(2,0,2,1)$ & 0 & 0 & 0 & 0 & 0 & 0 & 0 & 0 & 0 \\ \hline
$(2,0,1,2)$ & 0 & 0 & 0 & 0 & 0 & 0 & 0 & 0 & 0 \\ \hline
$(1,2,2,0)$ & 32 & 13 & 1 & 0 & 14 & 0 & 32 & 14 & 7 \\ \hline
$(1,2,1,1)$ & 24 & 14 & 4 & 0 & 12 & 0 & 12 & 6 & 3 \\ \hline
$(1,2,0,2)$ & 0 & 0 & 0 & 0 & 0 & 0 & 0 & 0 & 0 \\ \hline
$(1,1,2,1)$ & 24 & 14 & 4 & 0 & 12 & 0 & 12 & 6 & 3 \\ \hline
$(1,1,1,2)$ & 16 & 4 & 2 & 0 & 8 & 0 & 8 & 4 & 2 \\ \hline
$(1,0,2,2)$ & 0 & 0 & 0 & 0 & 0 & 0 & 0 & 0 & 0 \\ \hline
$(0,2,2,1)$ & 32 & 13 & 1 & 0 & 14 & 0 & 16 & 7 & 7 \\ \hline
$(0,2,1,2)$ & 16 & 1 & 0 & 0 & 8 & 0 & 8 & 4 & 4 \\ \hline
$(0,1,2,2)$ & 32 & 13 & 5 & 0 & 16 & 0 & 14 & 7 & 7 \\ \hline
\multicolumn{10}{|l|}{\textbf{total degree }$|\bd|=6$} \\ \hline
$(2,2,2,0)$ & 86 & 20 & 2 & 0 & 37 & 0 & 86 & 37 & 16 \\ \hline
$(2,2,1,1)$ & 64 & 26 & 10 & 0 & 32 & 0 & 32 & 16 & 7 \\ \hline
$(2,2,0,2)$ & 0 & 0 & 0 & 0 & 0 & 0 & 0 & 0 & 0 \\ \hline
$(2,1,2,1)$ & 48 & 7 & 2 & 0 & 24 & 0 & 24 & 12 & 6 \\ \hline
$(2,1,1,2)$ & 32 & 2 & 1 & 0 & 16 & 0 & 16 & 8 & 4 \\ \hline
$(2,0,2,2)$ & 0 & 0 & 0 & 0 & 0 & 0 & 0 & 0 & 0 \\ \hline
$(1,2,2,1)$ & 104 & 61 & 10 & 0 & 44 & 0 & 52 & 22 & 11 \\ \hline
$(1,2,1,2)$ & 48 & 7 & 2 & 0 & 24 & 0 & 24 & 12 & 6 \\ \hline
$(1,1,2,2)$ & 64 & 26 & 10 & 0 & 32 & 0 & 28 & 14 & 7 \\ \hline
$(0,2,2,2)$ & 86 & 20 & 2 & 0 & 37 & 0 & 37 & 16 & 16 \\ \hline
\multicolumn{10}{|l|}{\textbf{total degree }$|\bd|=7$} \\ \hline
$(2,2,2,1)$ & 288 & 100 & 21 & 0 & 120 & 0 & 144 & 60 & 26 \\ \hline
$(2,2,1,2)$ & 128 & 13 & 5 & 0 & 64 & 0 & 64 & 32 & 14 \\ \hline
$(2,1,2,2)$ & 128 & 13 & 5 & 0 & 64 & 0 & 56 & 28 & 14 \\ \hline
$(1,2,2,2)$ & 288 & 100 & 21 & 0 & 120 & 0 & 124 & 52 & 26 \\ \hline
\multicolumn{10}{|l|}{\textbf{total degree }$|\bd|=8$} \\ \hline
$(2,2,2,2)$ & 796 & 161 & 41 & 0 & 326 & 0 & 342 & 141 & 61 \\ \hline
\end{longtable}
\normalsize

\subsection*{Quivers with $5$ vertices}

\noindent Quiver (adjacency matrix):\\[4pt]
\renewcommand{\arraystretch}{1.5}
\begin{tabular}{l}
$Q_{1}=\left(\begin{smallmatrix}0&1&0&0&0\\0&0&1&0&0\\1&0&0&0&0\\0&0&0&0&1\\0&0&0&1&0\end{smallmatrix}\right)$ \\[6pt]
\end{tabular}
\renewcommand{\arraystretch}{1.0}

\small
\setlength{\tabcolsep}{6pt}
\begin{longtable}{|l|l|}
\caption{five-vertex Quivers}\label{tab:5vtx}\\
\hline
$\bd$ & $Q_{1}$ \\ \hline\hline
\endfirsthead
\multicolumn{2}{l}{\small\itshape Table~\ref{tab:5vtx} continued from previous page} \\ \hline
$\bd$ & $Q_{1}$ \\ \hline\hline
\endhead
\hline \multicolumn{2}{r}{\small\itshape continued on next page} \\
\endfoot
\hline
\endlastfoot
\multicolumn{2}{|l|}{\textbf{total degree }$|\bd|=0$} \\ \hline
$(0,0,0,0,0)$ & 0 \\ \hline
\multicolumn{2}{|l|}{\textbf{total degree }$|\bd|=1$} \\ \hline
$(1,0,0,0,0)$ & 1 \\ \hline
$(0,1,0,0,0)$ & 1 \\ \hline
$(0,0,1,0,0)$ & 1 \\ \hline
$(0,0,0,1,0)$ & 1 \\ \hline
$(0,0,0,0,1)$ & 1 \\ \hline
\multicolumn{2}{|l|}{\textbf{total degree }$|\bd|=2$} \\ \hline
$(1,1,0,0,0)$ & 1 \\ \hline
$(1,0,1,0,0)$ & 1 \\ \hline
$(1,0,0,1,0)$ & 0 \\ \hline
$(1,0,0,0,1)$ & 0 \\ \hline
$(0,1,1,0,0)$ & 1 \\ \hline
$(0,1,0,1,0)$ & 0 \\ \hline
$(0,1,0,0,1)$ & 0 \\ \hline
$(0,0,1,1,0)$ & 0 \\ \hline
$(0,0,1,0,1)$ & 0 \\ \hline
$(0,0,0,1,1)$ & 2 \\ \hline
\multicolumn{2}{|l|}{\textbf{total degree }$|\bd|=3$} \\ \hline
$(1,1,1,0,0)$ & 3 \\ \hline
$(1,1,0,1,0)$ & 0 \\ \hline
$(1,1,0,0,1)$ & 0 \\ \hline
$(1,0,1,1,0)$ & 0 \\ \hline
$(1,0,1,0,1)$ & 0 \\ \hline
$(1,0,0,1,1)$ & 0 \\ \hline
$(0,1,1,1,0)$ & 0 \\ \hline
$(0,1,1,0,1)$ & 0 \\ \hline
$(0,1,0,1,1)$ & 0 \\ \hline
$(0,0,1,1,1)$ & 0 \\ \hline
\multicolumn{2}{|l|}{\textbf{total degree }$|\bd|=4$} \\ \hline
$(1,1,1,1,0)$ & 0 \\ \hline
$(1,1,1,0,1)$ & 0 \\ \hline
$(1,1,0,1,1)$ & 0 \\ \hline
$(1,0,1,1,1)$ & 0 \\ \hline
$(0,1,1,1,1)$ & 0 \\ \hline
\multicolumn{2}{|l|}{\textbf{total degree }$|\bd|=5$} \\ \hline
$(1,1,1,1,1)$ & 0 \\ \hline
\end{longtable}
\normalsize

\bibliography{literature}
\end{document}